\documentclass[12pt]{amsart}
\usepackage[all]{xy}
\usepackage{a4wide}
\usepackage{amssymb}
\usepackage{amsthm}
\usepackage{amsmath}
\usepackage{amscd}
\usepackage{verbatim}
\usepackage[all]{xy}
\usepackage{framed,color}

\usepackage{graphicx}
\usepackage{subfig}

\usepackage{amsmath,amsthm,amssymb}
\usepackage{enumitem}

\usepackage{float}
\usepackage{ifthen}
\usepackage{xargs}
\usepackage{url}
\usepackage[all]{xy}
\usepackage{xcolor}
\usepackage{comment}
\usepackage[obeyDraft,textsize=tiny]{todonotes}
\specialcomment{draftnote}{\vspace{1em}\begingroup\ttfamily\footnotesize\color{blue}\begin{minipage}{15cm}}{\end{minipage}\endgroup\vspace{1em}}
\specialcomment{alert}{\vspace{1em}\begingroup\ttfamily\footnotesize\color{blue}\begin{minipage}{15cm}}{\end{minipage}\endgroup\vspace{1em}}

\usepackage[refpage,intoc,noprefix]{nomencl}

\def\({\left(}
\def\){\right)}

\def\lp{\left(}
\def\rp{\right)}

\newtheorem{theorem}{Theorem}[section]
\newtheorem{proposition}[theorem]{Proposition}

\theoremstyle{definition}

\newtheorem{remark}[theorem]{Remark}

\def\({\left(}
\def\){\right)}

\def\lp{\left(}
\def\rp{\right)}

\title{On the distribution of the norm of partitions}
\author{Walter Bridges and William Craig}

\begin{document}

\maketitle

\begin{abstract}
    The norm of an integer partition is defined as the product of its parts.  This statistic was recently introduced by Schneider in connection to partition zeta functions.  In this note, we use the method of moments to study the distribution of the norm under the uniform probability measure on partitions of $n$ as $n \to \infty$. We use singularity analysis to prove asymptotics for the moments and show as a result that the norm lacks a non-trivial limiting distribution on $[0,\infty)$.
\end{abstract}

\section{Introduction}

A {\it partition} of a natural number $n \geq 0$ is a non-increasing list $\lambda = \lp \lambda_1, \dots, \lambda_\ell \rp$ of non-negative integers such that $\lambda_1 + \lambda_2 + \dots + \lambda_\ell = n$, and we say $\lambda \vdash n$ when $\lambda$ is a partition of $n$. In this paper, we focus on the {\it norm} of a partition $\lambda \vdash n$, defined for $\lambda = \lp \lambda_1, \dots, \lambda_\ell \rp$ as the product of its parts, i.e.
\begin{align*}
    N\lp \lambda \rp := \prod_{j=1}^\ell \lambda_j.
\end{align*}
Interest in this new statistic comes from the work of Schneider and his collaborators \cite{ORR, ORW, R17, R18, R16} to develop a multiplicative theory of partitions that generalizes classical multiplicative number theory.  Here, the norm plays a central role in the construction of partition zeta functions \cite{R16}. For an overview of recent developments in this area, see \cite{ORR}.

The structure of large partitions is a central theme in the theory of partitions.  Celebrated work of Hardy--Ramanujan and Rademacher, using the modularity of the generating function for partitions, yields the well-known asymptotic for $p(n)$, the number of partitions of $n$:
\begin{align*}
    p(n) \sim \dfrac{1}{4n \sqrt{3}} e^{\pi \sqrt{\frac{2n}{3}}}.
\end{align*}
Another early result is the Erd\H{o}s--Lehner Theorem \cite[Theorem 1.1]{EL}, which gives the distribution for the largest part (or equivalently the number of parts) under the uniform measure on partitions of $n$, as $n \to \infty$.  A modern approach to many statistics for partitions was developed by Fristedt in \cite{Fristedt} and was subsequently generalized to a variety of combinatorial structures by Ariata and Tavar\'{e} \cite{AT}. To study the norm, however, we follow the more classical method of moments.

We first recall the maximum of the norm across partitions of $n$. It is given as an exercise in \cite[pp. 30--31]{Halmos} and \cite[p. 5]{Newman} and is restated in Schneider and Sill's recent overview of the norm \cite[Theorem 11]{SS}.

\begin{proposition}\label{P:SS}
    Let $n>1$.  Then we have
    $$
    M_n := \max\{N(\lambda) : \lambda \vdash n\} = \begin{cases} 3^{\frac{n}{3}} & \text{if $n \equiv 0 \pmod{3}$,} \\
    4 \cdot 3^{\frac{n-4}{3}} & \text{if $n \equiv 1 \pmod{3}$,} \\  2  \cdot 3^{\frac{n-2}{3}} & \text{if $n \equiv 2\pmod{3}$.} \\ \end{cases}
    $$
\end{proposition}

Let $E_n$ denote the expectation operator under the uniform measure on partitions of $n$.  For an integer $\ell \geq 1$, the $\ell$-th moment of the norm is
\begin{align*}
    E_n\lp N^\ell \rp = \dfrac{\sum_{\lambda \vdash n} N\lp \lambda \rp^\ell}{p(n)}.
\end{align*}
We use standard singularity analysis of the generating function for the $\ell$-th moments, to prove the following asymptotic formula.  As in Proposition \ref{P:SS}, the constants that appear to depend on $n \pmod{3}$.

\begin{theorem} \label{Moments theorem}
   Fix an integer $\ell \geq 1$.  For an integer $n$, write $n_0$ for the least residue of $n \pmod{3}$. Then as $n \to \infty$,
   $$
E_{n}(N^{\ell}) \sim c_{\ell}^{(n_0)}\frac{3^{\frac{\ell n}{3}}}{p(n)},
$$
where
$$
c_{\ell}^{(n_0)}:=\sum_{j=1}^3 e^{-\frac{2\pi i j n_0}{3}} \cdot \frac{1}{3}\prod_{\substack{k \geq 1 \\ k \neq 3}}\left(1-k^{\ell}3^{-\frac{k\ell}{3}}e^{\frac{2\pi i j k}{3}}\right)^{-1}.
$$
\end{theorem}

\begin{remark}
We make the following remarks on Theorem \ref{Moments theorem}.
\begin{enumerate}
    \item[(1)] Following our proof, one could give more terms in the asymptotic expansion above by taking more singularities into account.
    \item[(2)] Since the inequality $\textnormal{Var}\lp E_n \rp := E_n\lp N \rp^2 - E_n\lp N^2 \rp \geq 0$ fails for sufficiently large $n$, no rescaling of $N$ can converge to a non-trivial distribution on $[0,\infty)$. This lack of convergence is somewhat similar to certain hook-length statistics for partitions studied by Lang, Wan and Xu \cite{LWX}.
    \item[(3)] With a little extra work, one could obtain any finite-length asymptotic expansion for $E_n\lp N^\ell \rp$ using singularity analysis as in \cite[Chapter VI]{FS}. Indeed, it may be possible to use this method to compute moments of a wide range of multiplicative partition statistics.    See Schneider's constructions in \cite{R17, R16}.
\end{enumerate}
\end{remark}

\begin{table}[h]
\begin{tabular}{l|l|l|l}
    $\ell$ & $c^{(1)}_{\ell}$ & $c^{(2)}_{\ell}$  & $c^{(3)}_{\ell}$  \\ \hline
1 & $97922.939...$   & 97922.904...  & 97923.267...  \\ \hline
2 & $667.849...$   & 667.848... & 668.148...  \\ \hline
 3    &86.275...     & 86.298... & 86.602...   \\ \hline
 4    & 26.884...    & 26.927...  & 27.247...  \\ \hline
 5    & 12.453... & 12.513... & 12.851...  \\ \hline
  6   & 7.157...    & 7.229... &  7.588... \\ \hline
   7  & 4.689...    & 4.771... & 5.152... \\ \hline
    8 & 3.347...    & 3.435... & 3.839... \\ \hline
     9& 2.535...    & 2.628...  & 3.054... \\ \hline
         10 & 2.003...     & 2.100...  & 2.548... \\ \hline
\end{tabular}
\caption{The constants in Theorem \ref{Moments theorem}.} \label{T:1}
\end{table}

\begin{proof}[Proof of Theorem \ref{Moments theorem}]
    Let $\ell \geq 1$ be an integer. Using straightforward manipulations with generating series, it is clear that
    \begin{align*}
        F_\ell(q) := 1+ \sum_{n \geq 1} p(n)E_n\lp N^\ell \rp q^n = \sum_\lambda N\lp \lambda \rp^\ell q^{|\lambda|} = \prod_{k \geq 1} \dfrac{1}{1 - k^\ell q^k},
    \end{align*}
    where the sum $\sum\limits_\lambda$ runs over all partitions, regardless of size. As the series $\sum_{k \geq 1} k^{\ell}q^k$ converges for $|q|<1$, we see that the product defining $F_{\ell}$ converges absolutely for $|q|<1$, so long as $q$ avoids zeros of denominators in the product.  Hence, $F_{\ell}$ has singularities at 
    $$q=e^{\frac{2\pi i j}{k}}k^{-\frac{\ell}{k}}, \qquad \text{for $1\leq j \leq k,$}$$
    and is otherwise holormorphic for $|q|<1$.  Observe that $k \mapsto k^{-\frac{\ell}{k}}$ is a increasing function for $k \geq 3$, and for $k \in \mathbb{N}$ it obtains a minimum at $k=3$.  Thus, $F_{\ell}(q)$ has a meromorphic continuation to $|q|<2^{-\frac{1}{2}}$ with simple poles at $q=e^{\frac{2\pi i j}{3}}3^{-\frac{\ell}{k}}$ for $1 \leq j \leq 3$.

    We now apply singularity analysis as described in \cite[Chapter VI]{FS}.  As $q \to e^{\frac{2\pi i j}{3}}3^{-\frac{\ell}{3}}$, we have
    \begin{align*}
    F_{\ell}(q) &\sim \frac{1}{1-e^{-\frac{2\pi i j}{3}}3^{\frac{\ell}{3}}q} \cdot \frac{1}{(1-e^{-\frac{2\pi i (j+1)}{3}}e^{\frac{2\pi i j}{3}})(1-e^{-\frac{2\pi i (j+2)}{3}}e^{\frac{2\pi i j}{3}})}\prod_{\substack{k \geq 1 \\ k \neq 3}}\left(1-k^{\ell}3^{-\frac{k\ell}{3}}e^{\frac{2\pi i j k}{3}}\right)^{-1} \\
    &=\frac{1}{1-e^{-\frac{2\pi i j}{3}}3^{\frac{\ell}{3}}q} \cdot \frac{1}{3}\prod_{\substack{k \geq 1 \\ k \neq 3}}\left(1-k^{\ell}3^{-\frac{k\ell}{3}}e^{\frac{2\pi i j k}{3}}\right)^{-1}.
    \end{align*}
    Theorem 1.1 follows immediately from \cite[Theorem VI.5]{FS}.
\end{proof}

\end{document}